\documentclass{article}

\usepackage{color}
   \catcode`@=11
\usepackage{color}
\usepackage{amsthm,amssymb,amsfonts,mathrsfs}
\usepackage{amsmath}
\catcode`@=11 \@addtoreset{equation}{section}

\catcode`@=12

\newtheorem{thm}{Theorem}[section]
\newtheorem{proposition}{Proposition}[section]
\newtheorem{lemma}{Lemma}[section]
\newtheorem{corollary}{Corollary}[section]
\theoremstyle{definition}

\newtheorem{definition}{Definition}[section]
\newtheorem{remark}{Remark}

\newtheorem{hypothesis}{Hypothesis}

\def\R{{\mathbb{R}}}
\def\N{{\mathbb{N}}}
\newcommand{\cA}{{\mathcal A}}
\newcommand{\cY}{{\mathcal Y}}
\def\ds{\displaystyle}
\def\ve{\varepsilon}
\def\1e{\boldsymbol{1_\varepsilon}}

\usepackage[colorlinks=true,urlcolor=red,linkcolor=blue,citecolor=red,bookmarks=red]{hyperref}
\title{Stabilization for degenerate equations with drift and small singular term}
\author{
{\sc Genni Fragnelli$^a$, Dimitri Mugnai$^a$, Amine Sbai$^{b,c}$}\\\\
$^a$Dipartimento di Scienze Ecologiche e Biologiche,\\
Università della Tuscia,\\
Largo dell’Università, 01100 Viterbo, Italy,\\
email: (dimitri.mugnai, genni.fragnelli)@unitus.it\\\\
$^b$Hassan first University of Settat,\\
Faculty of Sciences and Technology, MISI Laboratory,\\
B.P 577, Settat 26000, Morocco.\\\\
$^c$Department of Applied Mathematics,\\
University of Granada, Granada, Spain.\\
email: a.sbai@uhp.ac.ma\\
}
\date{}
\begin{document}

\maketitle

\begin{abstract}
We consider a degenerate/singular wave equation in one dimension, with drift and in presence of a leading operator which is not in divergence form. We impose a homogeneous Dirichlet boundary condition where the degeneracy occurs and a boundary damping at the other endpoint. We provide some conditions for the uniform exponential decay of solutions for the associated Cauchy problem.\end{abstract}


%%%%%%%%%%%%%%%%%%%%%%%%%%%%%%%%%%%%%%%%%%%%%%%%%%%%
%%%%%%%%%%%%%%%%%%%%%%%%%%%%%%%%%%%%%%%%%%%%%%%%%%%%
\section{Introduction}\label{intro}

%%%%%%%%%%%%%%%%%%%%%%%%%%%%%%%%%%%%%%%%%%%%%%%%%%%%
%%%%%%%%%%%%%%%%%%%%%%%%%%%%%%%%%%%%%%%%%%%%%%%%%%%%
This paper is devoted to study the stabilization of degenerate equations of waves type with drift and with a small singular perturbation through a linear boundary feedback. To be more precise, we consider the following problem:
\begin{align}\label{mainequation}
\begin{cases}
y_{t t}-a(x) y_{x x}-b(x) y_x-\ds \frac{\lambda}{d(x)}y=0, & (t, x) \in Q, \\ y_t(t, 1)+\eta y_x(t, 1)+\beta y(t, 1)=0, & t \in(0, +\infty), \\ y(t, 0)=0, & t>0, \\ y(0, x)=y_0(x), \quad y_t(0, x)=y_1(x), & x \in(0,1),\end{cases}
\end{align}
where $Q=(0, +\infty) \times(0,1), a, b \in C^0[0,1]$, with $a>0$ on $(0,1], a(0)=0$ and $\ds\frac{b}{a} \in L^1(0,1)$\footnote{If $a(x)=x^K, K>0$, we can consider $b(x)=x^h$ for any $h>K-1$.}. In the boundary term we take $\beta \geq 0$ and $\eta$ is the well-known absolutely continuous weight function defined by 
\begin{equation}\label{eta}
\eta(x):=\exp\left \{\int_{\frac{1}{2}}^x\frac{b(s)}{a(s)}ds\right \}
, \quad x\in [0,1],
\end{equation}
introduced by Feller  several years ago (see \cite{F}). As usual (see, e.g.,  \cite{fmbook}), we find that $\eta\in C^0[0,1]\cap
C^1(0,1]$ is a strictly positive function and it can be extended to a function of class $C^1[0,1]$ when $b$ degenerates at 0 not slower than $a$ (for instance if $a(x)=x^K$ and $b(x)=x^h$,  $h\geq K$).

The fact that $a(0)=0$ introduces a lack of ellipticity in the spatial operator and for this reason the equation is {\it degenerate}. Degenerate equations have attracted great attention in the last decades, since they appear in several applied contexts, such as Physics (\cite{36}), Biology (\cite{9}, \cite{21}), Mathematical Finance (\cite{35}), Climatology (\cite{18}) and image processing  (\cite{17}). In particular, degenerate parabolic equations have been object of an intensive research in the last two decades (for instance, see \cite{acf,benoit,bfm,cfrjee,cfr,cavalcanti,fm13,fmmem,gao,wang}), mainly from the point of view of null controllability, in general obtained by new Carleman estimates. Problems in presence of both a degeneracy and a singularity have been considered, as well (see \cite{AFS,AHSS,att,BHSV,fjde,fm17,fmEJ})

A much less studied situation is the one of waves-type equations admitting a degeneracy term (see \cite{AFI,alabau,adv,bcg,BFM2022,lx,ZC}) or a singular term (see \cite{ca,VZ}), for which the authors prove controllability or (exponential) stability results.

As far as we know, for waves-type equations admitting simultaneously degeneracy and singularity, only controllability problems have been faced (see the recent papers \cite{ams,fms,MoS}), while nothing has been done for stability. For this reason, in this paper we focus on such a problem, proving that \eqref{mainequation} permits boundary stabilization, provided that the singular term has a small coefficient (see Theorem \ref{Energiastima} below). Hence, we may regard this result as a perturbation 
of the related one in \cite{fms}. However, the presence of the singular term $\frac{1}{d}u$ introduces several difficulties, which let us treat only the case of a function $d$ with {\it weak degeneracy}, according to the definition below.

\begin{definition}
A function $g$ is weakly degenerate at 0, (WD) for short, if $g \in$ $C^0[0,1] \cap C^1(0,1]$ is such that $g(0)=0, g>0$ on $(0,1]$ and, if
\begin{align}\label{WD}
\sup _{x \in(0,1]} \frac{x\left| g^{\prime}(x)\right|}{g(x)}=K_g,
\end{align}
then $K_g \in(0,1)$.
\end{definition}

\begin{definition}
 A function $g$ is strongly degenerate at 0, (SD) for short, if $g \in$ $C^1[0,1]$ is such that $g(0)=0, g>0$ on $(0,1]$ and in \eqref{WD} we have $K_g \in[1,2)$.
\end{definition}

Referring to Theorem \ref{Energiastima} for the precise assumptions and statement, a flavour of our exponential stability result is the following:
\begin{center}
{\sl If $\lambda$ is small and $a,b$ are not too degenerate, then\\ the energy of the solution to \eqref{mainequation} converges exponentially to 0 as time diverges.}
\end{center}
A case covered by our study is the one in which the the equation in \eqref{mainequation} is
\[
y_{t t}-x^\alpha y_{x x}-\mu x^\beta y_x-\ds \frac{\lambda}{x^\gamma}y=0.
\]
In this case, we have exponential stability provided that 
\[
\beta>0,\quad \beta>\alpha-1,\quad 2-\alpha-2\gamma> 2\mu \ \mbox{ and $\ \lambda$ small enough.}
\]

The paper is organized as follows. In Section \ref{sec2} we give the functional setting and some technical tools that we will use in the rest of the paper, together with the existence of solutions. In Section \ref{sec3} we introduce the energy associated to a solution of the problem and, by a multiplier method, we show that it decays exponentially as time diverges. In the Appendix we prove the existence theorem stated in Section \ref{sec2}.

%%%%%%%%%%%%%%%%%%%%%%%%%%%%%%%%%%%%%%%%%%%%%%%%%%%%
%%%%%%%%%%%%%%%%%%%%%%%%%%%%%%%%%%%%%%%%%%%%%%%%%%%%
\section{Preliminary results and well posedness}\label{sec2}
%%%%%%%%%%%%%%%%%%%%%%%%%%%%%%%%%%%%%%%%%%%%%%%%%%%%
%%%%%%%%%%%%%%%%%%%%%%%%%%%%%%%%%%%%%%%%%%%%%%%%%%%%
In this section we introduce the functional setting needed to face our problem, then we give some preliminary results about the function spaces we treat and finally we prove the well posedness result.

Following \cite{fms}, we start assuming a very modest requirement.
\begin{hypothesis}\label{hyp1}
The functions $a, b, d \in C^0[0,1]$ are such that 
\begin{enumerate}
\item $\ds \frac{b}{a} \in L^1(0,1)$,
\item $a(0)=d(0)=0$, $a,d>0$ on $(0,1]$,
\item there exist $K_1, K_2 \in(0,2)$ such that  the functions
\begin{equation}\label{ipoa}
x \longmapsto \frac{x^{K_1}}{a(x)}
\end{equation}
and
\begin{equation}\label{ipod}
x \longmapsto \frac{x^{K_2}}{d(x)}
\end{equation}
are nondecreasing in a right neighbourhood of $x=0$.
\end{enumerate}
\end{hypothesis}

\begin{remark}
If $a$ is (WD) or (SD), then \eqref{WD} implies that the function
\begin{align}\label{nondecreasxg/a}
x \mapsto \frac{x^\gamma}{a(x)} \text{ is nondecreasing in } (0,1] \text{ for all } \gamma \geq K_a.
\end{align}
In particular, \eqref{ipoa} holds and

\begin{align}\label{boundxgamba}
\left|\frac{x^\gamma b(x)}{a(x)}\right| \leq\tilde{M}
\end{align}
for all $\gamma \geq K_a$, where
\begin{align}\label{Mtilde}
\tilde{M}:=\frac{\|b\|_{L^{\infty}(0,1)}}{a(1)} .
\end{align}
Moreover,
\begin{align}\label{limitxgamma/a}
\lim _{x \rightarrow 0} \frac{x^\gamma}{a(x)}=0
\end{align}
for all $\gamma>K_a$.

 Analogously, if $d$ is (WD) or (SD), then
\begin{align}\label{nondecreasxg/d}
x \mapsto \frac{x^\gamma}{d(x)} \text{ is nondecreasing in } (0,1]
\end{align}
 for all  $\gamma \geq K_d$ and
\begin{align}\label{limitxgamma/da}
\lim _{x \rightarrow 0} \frac{x^\gamma}{d(x)}=0
\end{align}
for all $\gamma>K_d$.
\end{remark}

\begin{remark}
Concerning the notation, we will use $K_1$ ($K_2$, respectively) when Hypothesis \ref{hyp1} holds, and $K_a$ ($K_d$, respectively) when the global degeneracy conditions (WD) or (SD) hold true.
\end{remark}

Now set
\[
\sigma(x):=\frac{a(x)}{\eta(x)}, \quad x\in [0,1],
\]
and consider, as in \cite{fm},
the following Hilbert spaces with the related inner products
\[
 L^2_{\frac{1}{\sigma}}(0,1) :=\left\{ u \in L^2(0,1)\; \big|\; \int_0^1 \frac{u^2}{\sigma}dx <\infty \right\},
 \;  \langle u,v\rangle_{\frac{1}{\sigma}}:= \int_0^1u v\frac{1}{\sigma}dx
\]
for every $u,v \in L^2_{\frac{1}{\sigma}}(0,1)$ and
\[
H^1_{\frac{1}{\sigma}}(0,1) :=L^2_{\frac{1}{\sigma}}(0,1)\cap
H^1(0,1),
\;  \langle u,v\rangle_{1,\frac{1}{\sigma}}  :=   \langle u,v\rangle_{\frac{1}{\sigma}} + \int_0^1\eta  u'v'dx,\]  
for every $u,v \in H^1_{\frac{1}{\sigma}}(0,1)$. The previous inner products obviously induce the related norms
\[
\|u\|^2_{\frac{1}{\sigma}} = \int_0^1 \frac{u^2}{\sigma}dx \quad \text{and} \quad
\|u\|^2_{1,\frac{1}{\sigma}} = \|u\|^2_{\frac{1}{\sigma}} + \int_0^1\eta (u')^2 dx.
\]
Moreover, consider the space
\[
H^1_{\frac{1}{\sigma},0}(0,1) := \left\{ u \in H^1_{\frac{1}{\sigma}}(0,1) : u(0)=0\right\},
\]
endowed with the same inner product and norm of $H^1_{\frac{1}{\sigma}}(0,1)$.  Clearly, thanks to \eqref{ipoa}, and \eqref{ipod}, if $u \in H^1_{\frac{1}{\sigma},0}(0,1)$, then one can estimate the two integrals $\ds \int_0^1  \frac{u^2}{\sigma} dx$ and $\ds \int_0^1\frac{u^2}{\sigma d} dx$ with the norm of $u'$.
Indeed, as in \cite{BFM2022} and in \cite{fms}, the following Hardy-Poincar\'e-type inequalities hold.
\begin{proposition}\label{propL2}
Assume Hypothesis $\ref{hyp1}$. Then,
there exists
 $C>0$ such that for all $u \in H^1_{\frac{1}{\sigma},0}(0,1)$,
 \begin{equation}\label{stima1new}\int_0^1 u^2 \frac{1}{\sigma} dx
 \le C \int_0^1(u')^2dx.
 \end{equation}
Moreover, if $K_1+ K_2 \le 2$ and $u \in H^1_{\frac{1}{\sigma},0}(0,1)$ then $\ds\frac{u}{\sqrt{\sigma d}} \in L^2(0,1)$ and
there exists a positive constant $C>0$ such that
\begin{equation}\label{stima1}
\int_0^1 \frac{u^2}{\sigma d}dx \le C \int_0^1\eta (u'(x))^2dx.
\end{equation}
\end{proposition}
Let
\begin{equation}\label{CHP}
\tilde C_{HP}  \text{ and } C_{HP} \text{ be the best constants in \eqref{stima1new}  and \eqref{stima1},}
\end{equation}
respectively. Clearly, $\tilde C_{HP} \le  C_{HP}\max_{[0,1]}d \max_{[0,1]}\eta $. 
Moreover, under an additional assumption on $\lambda$, one can consider a different norm on $H^1_{\frac{1}{\sigma},0}(0,1) $. Assume the following:
\begin{hypothesis}\label{hyp2}
The constant $\lambda \in \R\setminus \{0\}$ is such that
\begin{equation}\label{lambda}
\lambda < \frac{1}{C_{HP}}.
\end{equation}
\end{hypothesis}
\noindent Observe that it is not restrictive to assume $\lambda \neq0$, since  in \cite{fm} the case $\lambda =0$ is considered.

If $K_1+ K_2 \le 2$ and if Hypotheses \ref{hyp1} and \ref{hyp2} hold, then one can  consider on
$
H^1_{\frac{1}{\sigma},0}(0,1) $ 
also the inner product
\[
\;  \langle u,v\rangle_1 :=   \langle u,v\rangle_{\frac{1}{\sigma}} + \int_0^1 \eta u'v'dx -\lambda \int_0^1 \frac{uv}{\sigma d},
\]
which induces the norm
\[
\|u\|_1^2:= \|u\|^2_{\frac{1}{\sigma}} + \int_0^1\eta (u')^2 dx-\lambda \int_0^1 \frac{u^2}{\sigma d},
\]
for all $
u, v \in H^1_{\frac{1}{\sigma},0}(0,1) $.
In particular, setting
\[
\|u\|_{1, \circ}^2:= \int_0^1 (u')^2dx\]
and
\[  \|u\|_{1, \bullet}^2:= \int_0^1\eta (u')^2dx
\]
for all $u \in H^1_{\frac{1}{\sigma},0}(0,1)$, one has the following equivalence thanks to Proposition \ref{propL2}:
\begin{corollary}
\label{equivalenze}Assume Hypotheses $\ref{hyp1}$ and $\ref{hyp2}$. If $K_1+ K_2\le 2$, then  the norms
$
\|\cdot\|_{1,\frac{1}{\sigma}}$, $\|\cdot\|_1$,
$
\|\cdot\|_{1, \circ}$ and  $\|\cdot\|_{1, \bullet}$
are equivalent in $H^1_{\frac{1}{\sigma},0}(0,1)$. 
\end{corollary}
Now define the operator
\begin{equation}\label{defsigma}
Ay:=ay''+by'= \sigma(\eta y')',
\end{equation}
for all $y \in D(A)$, where $D(A)$ is the Hilbert space
\[
D(A)=H^2_{\frac{1}{\sigma},0}(0,1) := \Big\{ u \in
H^1_{\frac{1}{\sigma},0}(0,1)\; \big|\;Au \in
L^2_{\frac{1}{\sigma}}(0,1)\Big\},
\]
with inner product
\[
 \langle u,v\rangle_{2, \frac{1}{\sigma}} := \langle u,v\rangle_{1, \frac{1}{\sigma}}+  \langle Au,Av\rangle_{\frac{1}{\sigma}}\]
 or
\[
 \langle u,v\rangle_{2} := \langle u,v\rangle_1+  \langle Au,Av\rangle_{\frac{1}{\sigma}}.\]

%\textcolor{red}{\begin{Remark}Notice that the previous result holds also if $u\in H^2_{{\frac{1}{\sigma}}}(0,1)$ wi
In the rest of the paper the next lemma is crucial; since its
proof is similar to the one of \cite[Lemma 3.2]{fms}, we omit it.
\begin{lemma}\label{lemmalimits}Assume Hypothesis $\ref{hyp1}.$
\begin{enumerate}
\item If $y \in D(A)$ and $u \in H^1_{\frac{1}{\sigma},0}(0,1)$, then $\lim\limits_{x \rightarrow 0} u(x) y'(x)=0$.
\item If  $u \in H^1_{\frac{1}{\sigma},0}(0,1)$, then   $\ds
\lim_{x\rightarrow 0} \frac{x}{a}u^2(x)=0.$
\item If  $u \in H^1_{\frac{1}{\sigma},0}(0,1)$ and $K_1+ K_2 <2$, then   $\ds
\lim_{x\rightarrow 0} \frac{x}{ad}u^2(x)=0.$
\item If  $u \in H^1_{\frac{1}{\sigma}}(0,1)$, then   $\ds
\lim_{x\rightarrow 0} \frac{x^2}{ad}u^2(x)=0.$
\item  If $u\in D(A)$, then $\ds\lim_{x\rightarrow 0} x^2  (u'(x))^2=0$.
\item If  $u \in D(A)$, $K_1 \le 1$, then $\displaystyle\lim_{x\rightarrow 0} x  (u'(x))^2=0$.
\item If  $u \in D(A)$,  $K_1>1$ and $\ds\frac{xb}{a} \in L^\infty(0,1), $ then $\ds\lim_{x\rightarrow 0} x  (u'(x))^2=0$.
\end{enumerate}
\end{lemma}
Thanks to the first point of the previous lemma one has the following integration by parts
\begin{lemma}\label{Lemma2.1} (\cite[Lemma 2.1]{fm}) 
Assume Hypothesis $\ref{hyp1}$. If 
$u\in H^2_{{\frac{1}{\sigma},0}}(0,1)$ and $v\in H^1_{{\frac{1}{\sigma},0}}(0,1)$, then
\[
 \langle Au,v\rangle_{\frac{1}{\sigma}}=-\int_0^1\eta u'v'dx+ [\eta u'v](1).
\]
\end{lemma}

Now, define 
\begin{equation}\label{Alambda}
A_\lambda u:= Au + \lambda\frac{u}{d}\end{equation}
for all $u \in D(A_\lambda):=\left\{u \in H^2_{\frac{1}{\sigma},0}(0,1): A_\lambda u\in L^2_{\frac{1}{\sigma}}(0,1)\right\}$. Clearly, $D(A_\lambda) \subseteq D(A)$; anyway, if $K_2 \le 1$, we have $\ds \frac{u}{d} \in L^2(0,1)$, and so
\[
D(A)=D(A_\lambda);
\]
however, in general such an inequality does not hold, unless $K_1+ 2K_2 \le 2$ (see \cite{fms}). Moreover, by Lemma \ref{Lemma2.1}, one has that if Hypothesis \ref{hyp1} is satisfied, then 
\begin{equation}\label{conal}
 \langle A_\lambda u,v\rangle_{\frac{1}{\sigma}}=-\int_0^1\eta u'v'dx+ \lambda\int_0^1\frac{u v}{\sigma d}dx+ [\eta u'v](1),
\end{equation}
for all $u\in H^2_{{\frac{1}{\sigma},0}}(0,1)$ and $v\in H^1_{{\frac{1}{\sigma},0}}(0,1)$. Thus, thanks to the new operator $(A_\lambda, D(A_\lambda))$, one can prove the next existence and regularity result, whose proof, being standard, is postponed to the Appendix.
\begin{thm}\label{thmmildclassolution}
Assume Hypotheses $\ref{hyp1}$ and $\ref{hyp2}$ with $K_1+ 2K_2 \le 2$. If $\left(y_0, y^1\right) \in H^1_{{\frac{1}{\sigma},0}}(0,1)\times L^2_{{\frac{1}{\sigma}}}(0,1)$, then there exists a unique mild solution
$$
y \in C^1\left([0,+\infty) ; L_{\frac{1}{\sigma}}^2(0,1)\right) \cap C\left([0,+\infty) ; H_{\frac{1}{\sigma},0}^1(0,1)\right)
$$
of \eqref{mainequation} which depends continuously on the initial data $\left(y_0, y_1\right) $. Moreover, if $\left(y_0, y_1\right) \in H^2_{{\frac{1}{\sigma},0}}(0,1) \times H^1_{{\frac{1}{\sigma},0}}(0,1)$, then the solution $y$ is classical, in the sense that
$$
y \in C^2\left([0,+\infty) ; L_{\frac{1}{\sigma}}^2(0,1)\right) \cap C^1\left([0,+\infty) ; H_{\frac{1}{\sigma},0}^1(0,1)\right) \cap C\left([0,+\infty); D(A_\lambda)\right)
$$
and the equation of \eqref{mainequation} holds for all $t \geq 0$.
\end{thm}
Observe that the assumption $K_1+2K_2\le 2$ implies $K_2 < 1$. This assumption is also crucial to prove the second part of the next variational result. The case which corresponds to a singular being strongly degenerate with $K_d \ge 1$ is an open problem.

The next  preliminary result will be crucial for the following. 
\begin{proposition}\label{Prop2.2} Assume Hypotheses $\ref{hyp1}$ and $\ref{hyp2}$ and for $\beta \ge 0$ define
\[
|||z|||_1^2:= \int_0^1 \eta (z')^2dx - \lambda \int_0^1 \frac{z^2}{\sigma d}dx +\beta z^2(1)
\] 
for all  $z \in H^1_{\frac{1}{\sigma},0}(0,1)$. Then the norms $|||\cdot|||_1$,  $||\cdot||_{1, \frac{1}{\sigma}}$, $||\cdot||_1$, $||\cdot||_{1, \bullet}$ and $\|\cdot\|_{1, \circ}$ are equivalent in $H^1_{\frac{1}{\sigma},0}(0,1)$. Moreover, for every $\gamma\in \R$, the variational problem
\begin{equation}\label{varionalproblem}
\int_0^1 \eta z'\phi' dx-\lambda \int_0^1 \frac{z\varphi}{\sigma d}+\beta z(1)\phi(1) =\gamma\phi(1), \quad \forall \; \phi \in H^1_{\frac{1}{\sigma},0}(0,1),
\end{equation}
admits a  unique solution $Z \in H^1_{\frac{1}{\sigma},0}(0,1)$ which satisfies the estimates
\begin{equation}\label{02}
|||Z|||_1^2\le \gamma^2C_\lambda^2 \quad \text{and} \quad \|Z\|^2_{{\frac{1}{\sigma}}}\le (\tilde C_{HP}+ \max_{[0,1]}\eta)\gamma^2C_\lambda^4,
\end{equation}
where $C_{HP},$  $\tilde C_{HP}$ are the best Hardy-Poincaré constants from Proposition $\ref{propL2}$ and
\begin{equation}\label{Clambda}
C_\lambda:= \begin{cases} \frac{1}{\sqrt{\min_{[0,1]}\eta}}, & \lambda \le0,\\
\frac{1}{\sqrt{\ve \min_{[0,1]}\eta}}, &  \lambda \in \left(0, \frac{1}{C_{HP}}\right),
\end{cases}
\end{equation}
where
$\ve \in (0,1)$ is such that
\begin{equation}\label{laepsi}
\lambda= \frac{1-\ve}{C_{HP}}.
\end{equation}
In addition, if $K_1+ 2K_2 \le 2$, then $Z \in D(A_\lambda)$ and solves
\begin{equation}\label{VP}
\begin{cases}
&-A_\lambda Z=0,\\
&\eta Z’(1) +\beta Z(1)=\gamma.
\end{cases}
\end{equation}
\end{proposition}
\begin{proof}To prove the equivalence among the norms, it is sufficient to prove that $|||\cdot|||_1$ and $||\cdot||_{1, \circ}$ are equivalent, thanks to Corollary \ref{equivalenze}.
Now, if $\lambda \le 0$, it is clear that for all $z\in H^1_{\frac{1}{\sigma},0}(0,1)$
\[\|z\|_{1,\circ}^2 \le \frac{1}{\min_{[0,1]} \eta}|||z|||^2_1.
\]
If $\lambda \in \left(0, \ds \frac{1}{C_{HP}}\right)$, we start fixing $\ve$ as indicated in \eqref{laepsi}. Hence, for all $z\in H^1_{\frac{1}{\sigma},0}(0,1)$, one has
\begin{equation}\label{startondino}
\int_0^1\eta (z')^2 dx -\lambda \int_0^1 \frac{z^2}{\sigma d} dx \ge \ve \int_0^1 \eta (z')^2dx
\end{equation}
and so
\[
\|z\|_{1,\circ}^2 \le \frac{1}{\ve \min_{[0,1]}\eta}|||z|||^2_1,
\]
for all  $z\in H^1_{\frac{1}{\sigma},0}(0,1)$. In conclusion, we have
\begin{equation}\label{starnew}
\|z\|_{1, \circ} \le C_\lambda |||z|||_1,
\end{equation}
where $C_\lambda$ is as in \eqref{Clambda}.
Now, we prove that there exists $C>0$ such that
\[|||z|||^2_1 \le C\|z\|_{1,\circ}^2,
\]
for all  $z\in H^1_{\frac{1}{\sigma},0}(0,1)$.  
First of all, observe that
\begin{equation}\label{*z}
|z(1)|= \left| \int_0^1 z'(t)dt \right| \le \|z\|_{1, \circ},
\end{equation} 
for all  $z\in H^1_{\frac{1}{\sigma},0}(0,1)$. Now let $\lambda> 0$. Clearly, if $z\in H^1_{\frac{1}{\sigma},0}(0,1)$,  from \eqref{*z}, we get
\[
|||z|||^2_1 \le (\max_{[0,1]}\eta + \beta)\|z\|_{1,\circ}^2.
\]
On the other hand, if  $\lambda <0$ and $z\in H^1_{\frac{1}{\sigma},0}(0,1)$, by \eqref{*z} we get
\[
|||z|||^2_1 \le (1-\lambda C_{HP}) \max_{[0,1]}\eta \int_0^1(z')^2dx + \beta z^2(1) \le  ((1-\lambda C_{HP})\max_{[0,1]}\eta + \beta)\|z\|_{1,\circ}^2,
\]
Thus, the two norms are equivalent.

Now, consider the bilinear and symmetric form  $\Lambda: H^1_{\frac{1}{\sigma},0}(0,1) \times H^1_{\frac{1}{\sigma},0}(0,1) \rightarrow \R$, given by
\[
\Lambda (z, \phi) := \int_0^1 \eta z' \phi' dx - \lambda \int_0^1 \frac{z\varphi}{\sigma d}dx+ \beta z(1)\phi(1).
\]
for all  $z,\phi \in H^1_{\frac{1}{\sigma},0}(0,1)$.
Clearly $\Lambda$ is coercive and continuous. Indeed, by Corollary \ref{equivalenze}
\begin{equation}\label{02bis}
\begin{aligned}
\Lambda(z,z) &= \int_0^1 \eta (z')^2dx - \lambda \int_0^1 \frac{z^2}{\sigma d}dx + \beta z^2(1)\\
&\ge  \int_0^1 \eta (z')^2dx- \lambda \int_0^1 \frac{z^2}{\sigma d}dx  \ge  \|z\|_{1,\bullet}^2
\end{aligned}
\end{equation}
if $\lambda \le 0$, and
\begin{equation}\label{02bis1}
\begin{aligned}
\Lambda(z,z) \ge  \left(1- \lambda C_{HP}\right)  \|z\|_{1,\bullet}^2,
\end{aligned}
\end{equation}
if $\lambda \in \left(0, \frac{1}{C_{HP}}\right)$.
Moreover, by the H\"older inequality and \eqref{*z},
\[
\begin{aligned}
|\Lambda (z, \phi)| &\le  \int_0^1 \eta |z' \phi'  |dx - \lambda \int_0^1 \left|\frac{z\phi }{\sigma d}\right|dx + \beta |z(1)||\phi(1)|\\
&\le ((1-\lambda C_{HP})\max_{x \in [0,1]}\eta + \beta)\|z\|_{1,\circ} \|\phi\|_{1,\circ}
\end{aligned}
\]
if $\lambda \le 0$, and
\[
|\Lambda (z, \phi)| \le (2\max_{x \in [0,1]}\eta + \beta)\|z\|_{1,\circ} \|\phi\|_{1,\circ}
\]
if $\lambda \in \left(0, \frac{1}{C_{HP}}\right)$.

Now, consider the linear functional $\mathcal L:H^1_{\frac{1}{\sigma},0}(0,1)\to \R$ defined as
\[
\mathcal L( \phi):= \gamma \phi(1),
\]
for every $\phi \in H^1_{\frac{1}{\sigma},0}(0,1)$. Clearly, $\mathcal L$ is continuous and linear. Thus, by the Lax-Milgram Theorem, there exists a unique solution $Z\in H^1_{\frac{1}{\sigma},0}(0,1)$ of
\begin{equation}\label{05}
\Lambda (Z, \phi)= \mathcal L (\phi)
\end{equation}
for all $\phi \in H^1_{\frac{1}{\sigma},0}(0,1)$.
In particular,
\begin{equation}\label{04}
|||Z|||_1^2=\Lambda (Z,Z) =\int_0^1 \eta (Z’)^2dx - \lambda \int_0^1 \frac{Z^2}{\sigma d}dx + \beta Z^2(1) =  \mathcal L(Z)= \gamma Z(1).
\end{equation}
By \eqref{starnew} and \eqref{*z},  we have that for all $z\in H^1_{\frac{1}{\sigma},0}(0,1)$
\begin{equation}\label{disclambda}
|z(1)| \le \|z\|_{1, \circ} \le C_\lambda |||z|||_1
\end{equation}
and, by \eqref{04},
\[
|||Z|||_1^2 = \gamma Z(1) \le\ds |\gamma|C_\lambda |||Z|||_1;\]
thus
\[
|||Z|||_1\le |\gamma|C_\lambda,\]
which proves the first part of \eqref{02}.
Moreover, by Proposition \ref{propL2}, \eqref{starnew} and \eqref{disclambda}, we have
\[
\begin{aligned}
\|Z\|^2_{{\frac{1}{\sigma}}}&\le \|Z\|^2_{1, \frac{1}{\sigma}}\le (\tilde C_{HP}+ \max_{[0,1]}\eta) \|Z\|^2_{1,\circ} \le (\tilde C_{HP}+ \max_{[0,1]}\eta)C_\lambda^2 |||Z|||^2_{1} \\
&\le (\tilde C_{HP}+ \max_{[0,1]}\eta)\gamma^2C_\lambda^4,
\end{aligned}
\]
thus concluding the validity of \eqref{02}.

Now, assume $K_1+ 2K_2\le 2$. We will prove that $Z \in H^2_{\frac{1}{\sigma},0}(0,1)$ and solves \eqref{VP}. To this aim, we consider again \eqref{05}. Since it holds for every $\phi \in H^1_{\frac{1}{\sigma},0}(0,1)$, it holds in particular for every $\phi \in C_c^\infty(0,1)$, so that 
\[
\int_0^1 \eta Z’\phi' - \lambda \int_0^1 \frac{Z\phi}{\sigma d}=0   \mbox{ for all }\phi \in C_c^\infty(0,1)\Longleftrightarrow  \int_0^1 \eta Z'\phi' =\lambda \int_0^1 \frac{Z\phi}{\sigma d}  \mbox{ for all }\phi \in C_c^\infty(0,1).
\]
Hence,
\[
( \eta Z’)’= - \lambda \frac{Z}{\sigma d} \quad \text{ a.e.  in} \; (0,1) \Longleftrightarrow -\sigma (\eta Z’)’ - \lambda \frac{Z}{d}=0 
\quad \text{ a.e.  in} \; (0,1)
\]
and so
$A_\lambda Z\in L^2_{\frac{1}{\sigma}}(0,1)$. Thanks to the assumption $K_1+ 2K_2\le 2$, one has that $\ds\frac{Z}{d} \in L^2_{\frac{1}{\sigma}}(0,1)$. Hence $A Z\in L^2_{\frac{1}{\sigma}}(0,1)$ and so $Z \in D(A_\lambda)$.

Finally, coming back to \eqref{05}, we have by \eqref{conal}
\[
\int_0^1 \eta Z’\phi' dx -\lambda \int_0^1 \frac{Z\varphi}{\sigma d} dx+ \beta Z(1)\phi(1)=\gamma \phi(1) \Longleftrightarrow (\eta Z'\phi)(1)+ \beta Z(1)\phi(1) =\gamma\phi(1)
\]
for all $\phi \in H^1_{\frac{1}{\sigma},0}(0,1)$. Thus, we obtain
\[
(\eta Z’)(1) + \beta Z(1)=\gamma;
\]
hence $Z$ solves \eqref{VP}.
\end{proof}

\begin{remark}
The previous result is similar to the one given in \cite[Proposition 2.2]{fm}, but here the presence of the singular term makes the situation more involved, and requires the condition $K_1+ 2K_2\le 2$. 
\end{remark}
%%%%%%%%%%%%%%%%%%%%%%%%%%%%%%%%%%%%%%%%%%%%%%%%%%%%
%%%%%%%%%%%%%%%%%%%%%%%%%%%%%%%%%%%%%%%%%%%%%%%%%%%%
\section{Stability result}\label{sec3}

%%%%%%%%%%%%%%%%%%%%%%%%%%%%%%%%%%%%%%%%%%%%%%%%%%%%
%%%%%%%%%%%%%%%%%%%%%%%%%%%%%%%%%%%%%%%%%%%%%%%%%%%%
In this section we prove the main result of this paper when $a$ and $d$ are (WD) or (SD). To this aim, if
 $y$ is a mild solution of \eqref{mainequation}, we consider its energy given by
$$
E_y(t)=\frac{1}{2}\left[\int_0^1\left(\frac{1}{\sigma} y_t^2(t, x)+\eta y_x^2(t, x) -\frac{\lambda}{\sigma d}  y^2(t, x) \right) d x+\beta y^2(t, 1)\right], \quad t \geq 0,
$$
and we prove that it decreases exponentially under suitable assumptions. In particular,
we collect all the previous assumptions in the following one:
\begin{hypothesis}\label{hyp3} Assume  that 
\begin{enumerate}
\item the functions $a$ and $d$ are (WD) or (SD), 
\item
$b \in C^0[0,1]$ is such that 
$\ds \frac{b}{a} \in L^1(0,1)$, 
\item $K_a+2K_d\le 2$,
\item if $K_a>1$, then $\ds \frac{x b}{a} \in L^{\infty}(0,1)$
\item $\lambda < \frac{1}{C_{HP}}.$
\end{enumerate}
\end{hypothesis}
We remark that the condition $\frac{xb}{a}\in L^\infty(0,1)$ is automatically verified if $K_a\leq1$ by \eqref{boundxgamba}. Moreover, we recall that the condition $K_a+2K_d\le 2$ (which specifies the previous $K_1+2K_2\leq 2$) implies that $d$ cannot be (SD). 
Under the previous hypotheses, one can prove the next result

\begin{thm}\label{energiadecrescente}
Assume Hypothesis $\ref{hyp3}$ and let $y$ be a classical solution of
\eqref{mainequation}. Then the energy is nonincreasing and
$$
\frac{d E_y(t)}{d t}=-y_t(t, 1)^2, \quad t \geq 0.
$$
\end{thm}
\begin{proof}
By multiplying the equation by $\ds \frac{y_t}{\sigma}$, integrating over $(0,1)$ and using the boundary conditions, one has
$$ 
0  =\frac{1}{2} \int_0^1 \frac{d}{d t}\left(\frac{y_t^2}{\sigma}\right) d x- \eta(1) y_x(t,1) y_t(t,1)+\int_0^1 \eta y_x y_{t x} d x -\int_0^1 \frac{\lambda}{\sigma d}  yy_t \, dx.
$$
Thus
$$ 
\begin{aligned}
0 & =\frac{1}{2} \frac{d}{d t}\left[\int_0^1\left(\frac{y_t^2}{\sigma}+\eta y_x^2 -\frac{\lambda}{\sigma d}  y^2 \right) d x+\beta y^2(t, 1)\right]+y_t^2(t, 1) \\
& = \frac{d}{d t} E_y(t)+y_t^2(t, 1) .
\end{aligned}
$$
Hence, the conclusion follows.
\end{proof}

Now, our aim is to estimate the energy at time $t>0$, $E_y(t)$, in terms of the energy at time $t=0$,  $E_y(0)$. To do that, we start with the following
\begin{proposition}
Assume Hypothesis $\ref{hyp3}$ and let $y$ be a classical solution of \eqref{mainequation}. Then, for all $T>s>0$
\begin{equation}\label{uguaglianza}
\begin{aligned}
0 & =2 \int_0^1\left[\frac{x y_x y_t}{\sigma}\right]_{t=s}^{t=T} d x-\frac{1}{\sigma(1)} \int_s^T y_t^2(t, 1) d t-\eta(1) \int_s^T y_x^2(t, 1) d t - \frac{\lambda}{\sigma(1)d(1)}\int_0^Ty^2(t,1)dt\\
&-\int_{Q_s} x \eta \frac{b}{a} y_x^2 d x d t +\int_{Q_s}\left(1-\frac{x\left(a^{\prime}-b\right)}{a}\right) \frac{1}{\sigma} y_t^2 d x d t+\int_{Q_s} \eta y_x^2 d x d t \\
&+ \lambda \int_{Q_s}\left( 1-\frac{x(a^\prime-b)}{a}-\frac{x d^\prime}{d} \right) \frac{y^2}{\sigma d}dxdt,
\end{aligned}
\end{equation}
where $Q_s:=(s, T) \times(0,1)$.
\end{proposition}
\begin{proof}
Take any $T>s>0$; then, multiplying the equation of \eqref{mainequation} by $\ds\frac{x y_x}{\sigma}$, integrating over $Q_s$ and recalling \eqref{defsigma}, we have, after some integration by parts,
$$
\begin{aligned}
0 & =\int_{Q_s} \frac{y_{t t} x y_x}{\sigma} d x d t-\int_{Q_s} x\left(\eta y_x\right)_x y_x d x d t - \lambda\int_{Q_s}\frac{x y y_x}{\sigma d} dxdt\\
& =\int_0^1\left[\frac{x y_x y_t}{\sigma}\right]_{t=s}^{t=T} d x-\frac{1}{2} \int_s^T\left[\frac{x}{\sigma} y_t^2\right]_{x=0}^{x=1} d t+\frac{1}{2} \int_{Q_s}\left(\frac{x}{\sigma}\right)^{\prime} y_t^2 d x d t-\int_{Q_s} x \eta \frac{b}{a} y_x^2 d x d t \\
& -\frac{1}{2} \int_s^T\left[x \eta y_x^2\right]_{x=0}^{x=1} d t+\frac{1}{2} \int_{Q_s}(x \eta)^{\prime} y_x^2 d x d t - \frac{ \lambda}{2}\int_s^T\left[\frac{xy^2}{\sigma d}\right]_{x=0}^{x=1} dt \\
&+ \frac{ \lambda}{2}\int_{Q_s}\left(\frac{\sigma d -x(\sigma^\prime d + \sigma d^\prime)}{(\sigma d)^2}\right) y^2dxdt.
\end{aligned}
$$
Recalling the definitions of $\eta$ and $\sigma$, we finally find
\begin{align}\label{ugual'}
0 & =\int_0^1\left[\frac{x y_x y_t}{\sigma}\right]_{t=s}^{t=T} d x-\frac{1}{2} \int_s^T\left[\frac{x}{\sigma} y_t^2\right]_{x=0}^{x=1} d t-\frac{1}{2} \int_s^T\left[x \eta y_x^2\right]_{x=0}^{x=1} d t - \frac{ \lambda}{2}\int_s^T\left[\frac{xy^2}{\sigma d}\right]_{x=0}^{x=1} dt \notag\\
& -\frac{1}{2} \int_{Q_s} x \eta \frac{b}{a} y_x^2 d x d t+\frac{1}{2} \int_{Q_s}\left(1-\frac{x\left(a^{\prime}-b\right)}{a}\right) \frac{1}{\sigma} y_t^2 d x d t+\frac{1}{2} \int_{Q_s} \eta y_x^2 d x d t \notag \\
&+ \frac{ \lambda}{2}\int_{Q_s}\left( 1-\frac{x(a^\prime-b)}{a}-\frac{x d^\prime}{d} \right) \frac{y^2}{\sigma d}dxdt.
\end{align}

Moreover, using the boundary conditions and Lemma \ref{lemmalimits}, we have that
$$
\lim _{x \rightarrow 0} \frac{x}{\sigma} y_t^2(t, x)=\lim _{x \rightarrow 0} \frac{x}{a} \eta y_t^2(t, x)=0,
$$
$$
\lim _{x \rightarrow 0} x \eta y_x^2(t, x)=0 
$$
and
$$
 \lambda \int_0^T\left[y^2 \frac{x}{\sigma  d}\right]_{x=0}^{x=1}=\frac{\lambda}{\sigma(1)d(1)}\int_0^T y^2(t,1)dt.
$$
Hence, \eqref{ugual'} multiplied by $2$ gives \eqref{uguaglianza}.
\end{proof}

\begin{proposition}
Assume Hypothesis $\ref{hyp3}$ and let $y$ be a classical solution of \eqref{mainequation}. Then, for all $T>s>0$ we have
\begin{align}\label{DT}
(B.T.) &= \int_{Q_s}\left(1-\frac{x\left(a^{\prime}-b\right)}{a}+\frac{K_a}{2}\right) \frac{1}{\sigma} y_t^2 d x d t+\int_{Q_s}\left(1-x \frac{b}{a}-\frac{K_a}{2}\right) \eta y_x^2 d x d t \\
& + \lambda \int_{Q_s}\left(1-\frac{x\left(a^{\prime}-b\right)}{a}-\frac{xd^\prime}{d}+\frac{K_a}{2}\right) \frac{y^2}{\sigma d} d x d t\nonumber
\end{align}
where
\begin{align}\label{BT}
(B.T.)=\int_0^1\left[-2 \frac{x y_x y_t}{\sigma}+\frac{K_a}{2} \frac{y y_t}{\sigma}\right]_{t=s}^{t=T} d x+\int_s^T\left[\frac{1}{\sigma} y_t^2+\eta y_x^2-\frac{K_a}{2} \eta y_x y+ \frac{\lambda}{\sigma d}y^2 \right](t, 1) d t.
\end{align} 
\end{proposition}
\begin{proof}
Multiplying the equation in \eqref{mainequation} by $\ds \frac{y}{\sigma}$ and integrating over $Q_s$, we have
\begin{align}\label{ugual''}
0=\int_{Q_s}\left(-\frac{y_t^2}{\sigma}+\eta y_x^2 -\lambda \frac{y^2}{\sigma d} \right) d x d t+\int_0^1\left[\frac{y y_t}{\sigma}\right]_{t=s}^{t=T} d x-\int_s^T\left[\eta y_x y\right]_{x=0}^{x=1} d t.
\end{align}
Using the boundary conditions and the fact that $y$ is a classical solution of \eqref{mainequation}, one has
$$
\int_s^T\left[\eta y_x y\right]_{x=0}^{x=1} d t=\int_s^T\left[\eta y_x y\right](t, 1) d t;
$$
thus, multiplying \eqref{ugual''} by $-\ds\frac{K_a}{2}$, we obtain
\begin{align}\label{ugual'''}
0=-\frac{K_a}{2} \int_{Q_s}\left(-\frac{y_t^2}{\sigma}+\eta y_x^2-\lambda \frac{y^2}{\sigma d} \right) d x d t-\frac{K_a}{2} \int_0^1\left[\frac{y y_t}{\sigma}\right]_{t=s}^{t=T} d x+\frac{K_a}{2} \int_s^T\left[\eta y_x y\right](t, 1) d t.
\end{align}
Summing \eqref{ugual'''} and \eqref{uguaglianza}, we get the claim.
\end{proof}

Now, define
\begin{equation}\label{1epsi}
\boldsymbol{1_\ve} := \begin{cases}1, & \lambda \le 0,\\
\ve, & \lambda=\frac{1-\ve}{C_{HP}}.
\end{cases}
\end{equation}
The following result holds:
\begin{proposition}\label{Prop3}
 Assume Hypothesis $\ref{hyp3}$, $\beta\ge 0$ and let $y$ be a classical solution of \eqref{mainequation}. Then, there exists $\1e \in(0,1]$, such that  for any $T>s>0$ and $\delta>0$ we have
\begin{equation}\label{Stimabo}
\begin{aligned}
\int_s^Ty^2(t,1)dt  &\le \left(\frac{2}{\boldsymbol{1_\ve} } \left(1+ \frac{\tilde C_{HP}C_\lambda^2}{\boldsymbol{1_\ve}\min_{[0,1]}^2\eta}\right)+ \frac{1}{2\delta} +\frac{1}{2\delta}(\tilde C_{HP} + \max_{[0,1]}\eta ) C_\lambda^4\right) E_y(s)\\
&+  \frac{\delta}{\boldsymbol{1_\ve}}\left(1+  \frac{1}{\boldsymbol{1_\ve}\min_{[0,1]}^2 \eta}C_\lambda^2 \right)\int_s^T E_y(t) dt,\end{aligned}
\end{equation}
where $C_\lambda$ is as in \eqref{Clambda}.
\end{proposition}
\begin{proof}The proof is similar to the one of Proposition 3.3 in \cite{fm}, so we will be sketchy.
Fix $t \in [s,T]$, set $\gamma =y(t,1)$ and let $z=z(t,\cdot)$ be the unique solution of 
\[
\int_0^1 \eta z'\phi' dx- \lambda \int_0^1 \frac{z\phi}{\sigma d}dx+ \beta z(1)\phi(1) = \gamma\phi(1), \quad \forall \; \phi \in H^1_{\frac{1}{\sigma},0}(0,1).
\]
By  Proposition \ref{Prop2.2}, $z(t,\cdot) \in D(A_\lambda)$ for all $t$ and solves
\begin{equation}\label{problem1}
\begin{cases}
&A_\lambda z=0,\\
&\eta z_x(t,1) +\beta z(t,1)=\gamma.
\end{cases}
\end{equation}
As in \cite{fm}, we
multiply the equation in \eqref{mainequation}  by $\ds \frac{z}{\sigma}$ and integrate over $Q_s$. Then, applying Lemma \ref{lemmalimits}, we have (notice that $z_t\in L^2_{\frac{1}{\sigma}}(0,1)$ by \eqref{02})
\begin{equation}\label{star1}
\int_0^1 \left[ y_t\frac{z}{\sigma}\right]_{t=s}^{t=T} dx- \int_{Q_s} y_t \frac{z_t}{\sigma}dxdt  - \lambda\int_{Q_s} \frac{y}{\sigma d}zdxdt= \int_s^T \eta y_x(t,1) z(1)dt -\int_{Q_s} \eta y_xz_xdxdt.
\end{equation}
By multiplying the equation in \eqref{problem1} by $\ds \frac{y}{\sigma}$, integrating on $Q_s$ and using the boundary conditions, one has
\[
\begin{aligned}  \int_s^T( \gamma - \beta z(t,1))y(t,1) dt = \int_{Q_s} \eta z_x y_x dxdt-  \lambda\int_{Q_s} \frac{y}{\sigma d}zdx dt.
\end{aligned}
\]
Substituting in \eqref{star1} and recalling that $y$ solves \eqref{mainequation}, so that, in particular, $\eta y_x(t,1)+\beta y(t,1)=-y_t(t,1)$, and $\gamma = y(t,1)$, we have
\begin{equation}\label{3pezzetti}
\int_s^Ty^2(t,1)dt = \int_{Q_s}\frac{y_tz_t}{\sigma}dxdt - \int_s^T(y_t z)(t,1)dt - \int_0^1 \left[\frac{y_t z}{\sigma}\right]_{t=s}^{t=T}dx.
\end{equation}
Hence, to estimate $\int_s^Ty^2(t,1)dt$, we have to consider the right-hand-side \eqref{3pezzetti}. 
As a first step, observe that, if $\lambda \le 0$, then
$$
\frac{1}{2}\int_0^1\left(\frac{1}{\sigma} y_t^2(t, x)+\eta y_x^2(t, x) \right)dx\le E_y(t) 
$$
%and
%\[
%\frac{1}{2}y^2(t,1) \le \frac{1}{2}\frac{1}{\min_{[0,1]}\eta} \int_0^1 (\eta y_x^2 )(t,x) dx \le \frac{1}{\min_{[0,1]}\eta} E_y(t),
%\]
for all $ t \in [s,T]$. If $\lambda \in \left(0, \ds \frac{1}{C_{HP}}\right)$ and $\ve \in (0,1)$ is as in \eqref{laepsi}, we obtain \eqref{startondino}, which implies that
\[
\begin{aligned}
2 E_y(t) &\ge \ve\int_0^1 \eta y_x^2(t,x)dx + \int_0^1 \frac{y_t^2(t,x)}{\sigma(x)}dx + \beta y^2(t,1)\\
&\ge \ve \left(\int_0^1 \eta y_x^2(t,x)dx + \int_0^1 \frac{y_t^2(t,x)}{\sigma(x)}dx \right) + \beta y^2(t,1);
\end{aligned}
\]
in particular,
\[
\frac{1}{2}\int_0^1\left(\frac{1}{\sigma} y_t^2(t, x)+\eta y_x^2(t, x) \right)dx\le \frac{1}{\ve}E_y(t)
\]
for all $t \ge0$.

Hence, for all considered $\lambda$ and all $t>0$, we have
\begin{equation}\label{stimae}
\frac{1}{2}\int_0^1\left(\frac{1}{\sigma} y_t^2(t, x)+\eta y_x^2(t, x) \right)dx\le \frac{1}{\1e}E_y(t)
\end{equation}
and
\begin{equation}\label{y(1)}
\frac{1}{2}y^2(t,1) = \frac{1}{2}\left(\int_0^1y_x(t,x)dx\right)^2\leq\frac{1}{2}\int_0^1 y_x^2(t,x) dx \le \frac{1}{2}\frac{1}{\min_{[0,1]}\eta} \int_0^1 (\eta y_x^2 )(t,x) dx \le \frac{1}{\1e \min_{[0,1]}\eta} E_y(t).
\end{equation}

Let us start estimating the third term in \eqref{3pezzetti}. 
To this aim, observe that
\begin{equation}\label{stima bo}
\int_0^1 (\eta z_x^2) \le \frac{1}{\boldsymbol{1_\ve}}|||z|||_1^2\end{equation}
by \eqref{02bis} and \eqref{02bis1}.
Hence, by applying Proposition \ref{propL2}, \eqref{stima bo}, \eqref{02}, \eqref{stimae}, \eqref{y(1)}, Theorem \ref{energiadecrescente}, and recalling that $\gamma= y(t,1)$, we have 
\[
\begin{aligned}
\int_0^1\left| \frac{y_t z}{\sigma}(\tau, x)\right|dx&\le \frac{1}{2} \int_0^1 \frac{y_t^2(\tau, x)}{\sigma} dx + \frac{1}{2} \int_0^1 \frac{z^2(\tau, x)}{\sigma}dx \\
&\le  \frac{1}{2}\int_0^1 \frac{y_t^2(\tau, x)}{\sigma} dx +\frac{1}{2\min_{[0,1]}\eta}\tilde C_{HP}\int_0^1 (\eta z_x^2)(\tau, x) dx \\
& \le \frac{1}{2} \int_0^1 \frac{y_t^2(\tau, x)}{\sigma} dx +\frac{\tilde C_{HP}C_\lambda^2}{\boldsymbol{1_\ve}2\min_{[0,1]}\eta}y^2(t,1)\\
& \le \frac{1}{\1e} E_y(\tau) +\frac{\tilde C_{HP}C_\lambda^2}{\1e^2(\min_{[0,1]} \eta)^2}E_y(t)\\
& \le \frac{1}{\1e} \left( 1+ \frac{\tilde C_{HP}C_\lambda^2}{\boldsymbol{1_\ve}(\min_{[0,1]} \eta)^2}\right) E_y(s)
\end{aligned}
\]
for all $\tau \in [s,T]$. Here $C_\lambda$ is the constant defined in \eqref{Clambda}.
Thus, again by Theorem \ref{energiadecrescente}, we have
\begin{equation}\label{terzo3pez}
\left|  \int_0^1 \left[\frac{y_t z}{\sigma}\right]_{t=s}^{t=T}dx \right| \le \frac{2}{\1e} \left( 1+ \frac{\tilde C_{HP}C_\lambda^2}{\boldsymbol{1_\ve}(\min_{[0,1]} \eta)^2}\right) E_y(s).
\end{equation}

Now, consider the second term in \eqref{3pezzetti}.
Then, for any $\delta >0$ we have
\begin{equation}\label{7.1}
\int_s^T |(y_tz)(t,1) |dt \le \frac{1}{2\delta} \int_s^T  y_t^2 (t,1)dt + \frac{\delta}{2} \int_s^T z^2(t,1)dt.
\end{equation}

Starting as in \eqref{y(1)}, by  \eqref{02}, \eqref{stima bo} and \eqref{y(1)}, we have
\[
\begin{aligned}
z^2(t,1) &\le \frac{1}{\min_{[0,1]}\eta}\int_0^1 (\eta z_x^2)(t,x) dx \le \frac{1}{\1e\min_{[0,1]}\eta}|||z|||_1^2 \\
&\le \frac{1}{\1e\min_{[0,1]}\eta} C_\lambda^2 y^2(t,1)\le \frac{2}{\1e^2\min_{[0,1]}^2 \eta}C_\lambda^2  E_y(t).
\end{aligned}
\]
Thus, by Theorem \ref{energiadecrescente}, we have
\begin{equation}\label{secondo3pez}
\begin{aligned}
\int_s^T |(y_tz)(t,1)| dt &\le \frac{1}{2\delta} \int_s^T  y_t^2 (t,1)dt + \frac{\delta}{2} \int_s^T z^2(t,1)dt\\
& \le - \frac{1}{2\delta} \int_s^T \frac{dE_y(t)}{dt} dt + \delta \frac{1}{\1e^2\min_{[0,1]}^2 \eta}C_\lambda^2   \int_s^T E_y(t) dt  \\
& \le \frac{E_y(s)}{2\delta} +  \delta \frac{1}{\1e^2\min_{[0,1]}^2 \eta}C_\lambda^2   \int_s^T E_y(t) dt.
\end{aligned}
\end{equation}

Finally, we estimate $\ds\int_{Q_s}\frac{y_tz_t}{\sigma}dxdt $. To this aim, consider again \eqref{problem1} and differentiate with respect to $t$.
 Then
\[
\begin{cases}
&- \sigma (\eta z_{tx})_x - \lambda \ds\frac{z_t}{d}=0,\\
&\eta z_{tx}(t,1) +\beta z_t(t,1)=\gamma_t =y_t(t,1).
\end{cases}
\]
By \eqref{02}, $z_t$ satisfies the estimates
\[
|||z_t|||^2_1 \le C_\lambda^2 y_t^2(t,1) \quad \text{and} \quad \|z_t\|^2_{{\frac{1}{\sigma}}} \le (\tilde C_{HP} + \max_{[0,1]}\eta ) C_\lambda^4y_t^2(t,1).
\]
Thus, by \eqref{stimae}, Theorem \ref{energiadecrescente} we find
\begin{equation}\label{primo3pez}
\begin{aligned}
\int_{Q_s}\left|\frac{y_t z_t}{\sigma}\right|dxdt &\le \frac{\delta}{2} \int_{Q_s} \frac{y_t^2}{\sigma}dxdt + \frac{1}{2\delta}\int_{Q_s} \frac{z_t^2}{\sigma}dxdt\\
&\le \frac{\delta}{\1e} \int_s^TE_y(t)dt + \frac{1}{2\delta}(\tilde C_{HP} + \max_{[0,1]}\eta ) C_\lambda^4\int_s^T y_t^2(t,1)dt\\
&= \frac{\delta}{\1e} \int_s^TE_y(t)dt - \frac{1}{2\delta}(\tilde C_{HP} + \max_{[0,1]}\eta ) C_\lambda^4\int_s^T\frac{dE_y(t)}{dt}dt\\
& \le  \frac{\delta}{\1e} \int_s^TE_y(t)dt+ \frac{1}{2\delta}(\tilde C_{HP} + \max_{[0,1]}\eta ) C_\lambda^4E_y(s).
\end{aligned}
\end{equation}
Hence, going back to \eqref{3pezzetti}, by \eqref{primo3pez}, \eqref{secondo3pez} and \eqref{terzo3pez} we get
\[
\begin{aligned}
\int_s^Ty^2(t,1)dt  &\le \left(\frac{2}{\1e} + \frac{2}{\1e^2} \frac{\tilde C_{HP}C_\lambda^2}{\min_{[0,1]}^2\eta}+ \frac{1}{2\delta} +\frac{1}{2\delta}(\tilde C_{HP} + \max_{[0,1]}\eta ) C_\lambda^4\right) E_y(s)\\
&+  \frac{\delta}{\1e}\left(1+  \frac{1}{\1e\min_{[0,1]}^2 \eta}C_\lambda^2 \right)\int_s^T E_y(t) dt,\end{aligned}
\]
as claimed.
\end{proof}

As in \cite{fm}, we assume an additional hypothesis on functions $a$ and $b$. Clearly, this assumption is stronger with respect to the one in \cite{fm} since here we have also a singular term to manage.
\begin{hypothesis}\label{Ass2}
Hypothesis \ref{hyp3} holds and there exists $\ve_0>0$ such that $(2-K_a-2K_d)a-2x|b|\geq \ve_0a$ for every $x\in [0,1]$.
\end{hypothesis}
\begin{proposition}\label{Prop4}
 Assume Hypothesis $\ref{Ass2}$ and let $y$ be a classical solution of \eqref{mainequation}. Then, for any $T>s>0:$
\newline
\begin{equation}\label{tondonew}
\begin{aligned}
\frac{\ve_0}{2} \int_{Q_s} \left(\frac{y_t^2}{\sigma} + \eta y_x^2 - \lambda\frac{y^2}{\sigma d}  \right)dxdt
&\le \left( 2\Theta +\frac{1}{\sigma(1)} + \frac{1}{\eta(1)}+\frac{\beta}{\eta(1)}+ \frac{K_a}{4}\right)E_y(s) 
\\
&+  \left(\frac{\beta^2}{\eta(1)}+ \frac{K_a\beta}{2}+\frac{\beta}{\eta(1)}+ \frac{K_a}{4}+ \frac{\lambda}{\sigma(1)d(1)} \right) \int_s^Ty^2(t,1)dt,
\end{aligned}
\end{equation}
if $\lambda \ge 0$, and
\begin{equation}\label{tondonewbis}
\begin{aligned}
\frac{\ve_0}{2} \int_{Q_s} \left(\frac{y_t^2}{\sigma} + \eta y_x^2 - \lambda\frac{y^2}{\sigma d}  \right)dxdt & \le\left( 2\Theta +\frac{1}{\sigma(1)} + \frac{1}{\eta(1)}+\frac{\beta}{\eta(1)}+ \frac{K_a}{4} \right)E_y(s) 
\\&+  \left(\frac{\beta^2}{\eta(1)}+ \frac{K_a\beta}{2}+\frac{\beta}{\eta(1)}+ \frac{K_a}{4} \right) \int_s^Ty^2(t,1)dt\\
& -2\lambda C_{HP} \left(1+ \frac{3}{2}K_a + K_d+M\right) \int_{s}^T E_y(t)dt,
\end{aligned}
\end{equation}
if $\lambda <0$.
Here 
\begin{equation} \label{Theta}\ds \Theta:= \frac{2}{\1e}  \max\left\{\frac{1}{a(1)}+ \frac{K_a\tilde C_{HP}}{4\min_{[0,1]}\eta },1+ \frac{K_a}{4} \right\},
\end{equation}
 $\1e$ is given in \eqref{1epsi} and
\[
M:= \begin{cases}
\tilde{M}, & K_a \le 1,\\
\ds \left\| \frac{xb}{a}\right\|_{L^\infty(0,1)}, & K_a >1,
\end{cases}
\]
where $\tilde{M}$ is defined in \eqref{Mtilde}.
\end{proposition}
\begin{proof}
Thanks to Hypothesis \ref{Ass2} we have
\begin{equation}\label{3infila}
\begin{aligned}
&1-\frac{x(a'-b)}{a}+\frac{K_a}{2}\geq \frac{(2-K_a-2K_d)a+2K_da-2x|b|}{2a}\geq \frac{\ve_0}{2}+K_d>\frac{\ve_0}{2},\\&
1-x\frac{b}{a}-\frac{K_a}{2}\geq \frac{(2-K_a-2K_d)a+2K_da-2x|b|}{2a}\geq \frac{\ve_0}{2}+K_d>\frac{\ve_0}{2},\\
&
1-\frac{x\left(a^{\prime}-b\right)}{a}-\frac{xd^\prime}{d}+\frac{K_a}{2}\ge \frac{(2-2K_a)a  -2x|b|+K_aa-2K_da}{2a} \ge \frac{\ve_0}{2}.
\end{aligned}
\end{equation}
Now, we distinguish between the case $\lambda \ge0$ and $\lambda <0$.

\vspace{0.3cm}
\noindent {{\bf Case} $\boldsymbol{\lambda \ge0}$}. In this case, the distributed terms in \eqref{DT} can be estimated from below in the following way:
\begin{equation}\label{star}
\begin{aligned}
&\int_{Q_s}\left(1-\frac{x\left(a^{\prime}-b\right)}{a}+\frac{K_a}{2}\right) \frac{1}{\sigma} y_t^2 d x d t+\int_{Q_s}\left(1-x \frac{b}{a}-\frac{K_a}{2}\right) \eta y_x^2 d x d t \\
& + \lambda \int_{Q_s}\left(1-\frac{x\left(a^{\prime}-b\right)}{a}-\frac{xd^\prime}{d}+\frac{K_a}{2}\right) \frac{y^2}{\sigma d} d x d t\\
&\ge  \frac{\ve_0}{2} \int_{Q_s} \left(\frac{y_t^2}{\sigma} + \eta y_x^2 - \lambda\frac{y^2}{\sigma d}  \right)dxdt.
\end{aligned}
\end{equation}
Now, we estimate the boundary terms in \eqref{BT} from above.  First of all, consider the term 
\[
\int_0^1\left( -2x\frac{y_xy_t}{\sigma} + \frac{K_a}{2} \frac{yy_t}{\sigma}\right) (\tau, x) dx
\]
for all $\tau \in [s,T]$. By \eqref{nondecreasxg/a}, one has that $\ds \frac{x^2}{a(x)} \le \frac{1}{a(1)}$; using this fact together with Proposition \ref{propL2} and \eqref{stimae}, one has
\[
\begin{aligned}
\int_0^1&\left( -2x\frac{y_xy_t}{\sigma} + \frac{K_a}{2} \frac{yy_t}{\sigma}\right)(\tau, x) dx \\
&\le \int_0^1 \frac{x^2y_x^2\eta}{a}(\tau, x) dx+ \int_0^1\frac{y_t^2}{\sigma} (\tau, x)dx+ \frac{K_a}{4}\int_0^1 \frac{y_t^2}{\sigma}(\tau, x)dx + \frac{K_a}{4}\int_0^1 \frac{y^2}{\sigma}(\tau, x)dx\\
&\le \frac{1}{a(1)}\int_0^1 \eta y_x^2(\tau, x) dx + \left(1 +\frac{K_a}{4}\right)\int_0^1 \frac{y_t^2}{\sigma}(\tau, x)dx +  \frac{K_a\tilde C_{HP}}{4\min_{[0,1]}\eta }\int_0^1\eta y_x^2 (\tau, x)dx\\
&\le \frac{2}{\1e}  \max\left\{\frac{1}{a(1)}+ \frac{K_a\tilde C_{HP}}{4\min_{[0,1]}\eta },1+ \frac{K_a}{4} \right\}E_y(\tau),
\end{aligned}
\]
for all $\tau \in [s, T]$.
In particular,
\begin{equation}\label{stimay}
\begin{aligned}
&\int_0^1\left[\left( -2x\frac{y_xy_t}{\sigma} + \frac{K_a}{2} \frac{yy_t}{\sigma}\right)(\tau, x) dx \right]_{\tau=s}^{\tau=T} \le 2\Theta E_y(s),
\end{aligned}
\end{equation}
where $\ds \Theta$ is defined in \eqref{Theta}. By \eqref{DT}, \eqref{star} and the previous inequality, we get
\begin{equation}\label{tondo}
\begin{aligned}
\frac{\ve_0}{2} \int_{Q_s} \left(\frac{y_t^2}{\sigma} + \eta y_x^2 - \lambda\frac{y^2}{\sigma d}  \right)dxdt&\le 2\Theta E_y(s) +\int_s^T\left[\frac{1}{\sigma} y_t^2+\eta y_x^2-\frac{K_a}{2} \eta y_x y+ \frac{\lambda}{\sigma d}y^2 \right](t, 1) d t.\end{aligned}
\end{equation}
As \cite[Equation (3.22)]{fm}, one can prove that
\begin{equation}\label{tondo1}
\begin{aligned}
\int_s^T \left[\frac{1}{\sigma} y_t^2+\eta y_x^2-\frac{K_a}{2} \eta y_x y+ \frac{\lambda}{\sigma d}y^2 \right](t, 1)dt &\le  \left( \frac{1}{\sigma(1)} + \frac{1}{\eta(1)}+\frac{\beta}{\eta(1)}+ \frac{K_a}{4} \right)(E_y(s)-E_y(T)) \\
&+  \left(\frac{\beta^2}{\eta(1)}+ \frac{K_a\beta}{2}+\frac{\beta}{\eta(1)}+ \frac{K_a}{4} + \frac{\lambda}{\sigma(1)d(1)}\right) \int_s^Ty^2(t,1)dt.
\end{aligned}
\end{equation}
Hence, by \eqref{stimay}, \eqref{tondo} and \eqref{tondo1}, we have
\[
\begin{aligned}
&\frac{\ve_0}{2} \int_{Q_s} \left(\frac{y_t^2}{\sigma} + \eta y_x^2 - \lambda\frac{y^2}{\sigma d}  \right)dxdt\\
&
\le \left( 2\Theta +\frac{1}{\sigma(1)} + \frac{1}{\eta(1)}+\frac{\beta}{\eta(1)}+ \frac{K_a}{4}\right)E_y(s) 
\\
&+  \left(\frac{\beta^2}{\eta(1)}+ \frac{K_a\beta}{2}+\frac{\beta}{\eta(1)}+ \frac{K_a}{4} + \frac{\lambda}{\sigma(1)d(1)}\right) \int_s^Ty^2(t,1)dt,
\end{aligned}
\]
as claimed.

\vspace{0.3cm}
\noindent {{\bf Case} $\boldsymbol{\lambda \le0}$}. In this case, by definition of energy, one has
\[
-\lambda\int_{Q_s}\frac{y^2}{\sigma d} d x d t\le -2\lambda C_{HP} \int_s^T E_y(t)dt;\footnote{Of course, a straightforward estimate could be $-\lambda\int_{Q_s}\frac{y^2}{\sigma d} d x d t\le 2 \int_s^T E_y(t)dt$, but we could not manage such a term in Theorem \ref{Energiastima}.}
\]
thus, by \eqref{3infila}, \eqref{DT} and by estimating the boundary terms as in the previous case (with the exception that in \eqref{tondo1} the term with $\lambda$ disappears), we find
\[
\begin{aligned}
&\frac{\ve_0}{2} \int_{Q_s} \left(\frac{y_t^2}{\sigma} + \eta y_x^2 - \lambda\frac{y^2}{\sigma d}  \right)dxdt \le \int_{Q_s}\left(1-\frac{x\left(a^{\prime}-b\right)}{a}+\frac{K_a}{2}\right) \frac{1}{\sigma} y_t^2 d x d t+\int_{Q_s}\left(1-x \frac{b}{a}-\frac{K_a}{2}\right) \eta y_x^2 d x d t \\
& \le \int_0^1\left[\left( -2x\frac{y_xy_t}{\sigma} + \frac{K_a}{2} \frac{yy_t}{\sigma}\right)(\tau, x) dx \right]_{\tau=s}^{\tau=T} +\int_s^T\left[\frac{1}{\sigma} y_t^2+\eta y_x^2-\frac{K_a}{2} \eta y_x y+ \frac{\lambda}{\sigma d}y^2 \right](t, 1) d t\\
&-
 \lambda \int_{Q_s}\left(1-\frac{x\left(a^{\prime}-b\right)}{a}-\frac{xd^\prime}{d}+\frac{K_a}{2}\right) \frac{y^2}{\sigma d} d x d t\\
 & \le\left( 2\Theta +\frac{1}{\sigma(1)} + \frac{1}{\eta(1)}+\frac{\beta}{\eta(1)}+ \frac{K_a}{4} \right)E_y(s) 
+  \left(\frac{\beta^2}{\eta(1)}+ \frac{K_a\beta}{2}+\frac{\beta}{\eta(1)}+ \frac{K_a}{4} \right) \int_s^Ty^2(t,1)dt\\
& -2\lambda C_{HP} \left(1+ \frac{3}{2}K_a + K_d+M\right)\int_{s}^T E_y(t)dt.
\end{aligned}
\]
Our proof is then concluded.
\end{proof}
As a consequence of Propositions \ref{Prop3} and \ref{Prop4}, we have the main result of the paper. As a first step, define
\[C_1:= 2\Theta +\frac{1}{\sigma(1)} + \frac{1}{\eta(1)}+\frac{\beta}{\eta(1)}+ \frac{K_a}{4} ,
\]
\[
C_2:=\begin{cases}\frac{\beta^2}{\eta(1)}+ \frac{K_a\beta}{2}+\frac{\beta}{\eta(1)}+ \frac{K_a}{4}+ \frac{\beta \ve_0}{2}+\frac{\lambda}{\sigma(1)d(1)}, & \lambda \ge 0,\\
\frac{\beta^2}{\eta(1)}+ \frac{K_a\beta}{2}+\frac{\beta}{\eta(1)}+ \frac{K_a}{4}+ \frac{\beta \ve_0}{2}, & \lambda <0,
\end{cases}
\]
\[C_3:=\frac{2}{\1e} +  \frac{2}{\1e^2\min_{[0,1]}^2\eta}\tilde C_{HP}C_\lambda^2+ \frac{1}{2\delta} +\frac{1}{2\delta}(\tilde C_{HP} + \max_{[0,1]}\eta ) C_\lambda^4
\]
and
\[C_4:=\frac{1}{\1e}\left(1+  \frac{1}{\1e\min_{[0,1]}^2 \eta}C_\lambda^2 \right)\]
where
$\delta >0$.
\begin{thm}\label{Energiastima}
 Assume Hypothesis $\ref{Ass2}$ and if $\lambda <0$, then 
 $\lambda \in \left(- \frac{\ve_0}{2C_{HP}\left(1+ \frac{3}{2}K_a + K_d+M\right)}, 0\right)$. Let $y$ be a mild solution of \eqref{mainequation}. Then, for all $t \ge \mathcal M$ and for all $\ds\delta \in \left(0, \delta_0\right)$,
\begin{equation}\label{estiener}
E_y(t) \le E_y(0) e^{1- \frac{t}{\mathcal M}},
\end{equation}
where $
\delta_0:= \min\left\{\frac{\ve_0}{C_2C_4},  \frac{\ve_0+2\lambda_{HP}\left(1+ \frac{3}{2}K_a + K_d+M\right)}{C_2C_4}\right\}
$ and
\begin{equation}\label{M}
\mathcal M:=\begin{cases}\frac{C_1+ C_2C_3}{\ve_0 -  C_2C_4\delta}, & \lambda \in \left[0,\frac{1}{C_{HP}}\right),\\
\frac{C_1+ C_2C_3}{\ve_0 -  C_2C_4\delta+2\lambda C_{HP} \left(1+ \frac{3}{2}K_a + K_d+M\right)},& \lambda \in \left(- \frac{\ve_0}{2C_{HP}\left(1+ \frac{3}{2}K_a + K_d+M\right)}, 0\right).
\end{cases}
\end{equation}
\end{thm}
\begin{proof} 
As usual, let us start assuming that $y$ is a classical solution. 
If $y$ is the mild solution associated to the initial data $(y_0, y_1) \in \mathcal H_0$, consider a sequence $\{(y_0^n, y_1^n)\}_{n \in  \N} \in D(\mathcal A)$ that approximate $(y_0, y_1)$ and let $y^n$ be the classical solution of \eqref{mainequation} associated to $(y_0^n, y_1^n)$. With standard estimates coming from Theorem \ref{generator}, we can pass to the limit in \eqref{estiener} written for $y^n$ and obtain the desired result.
So, let $y$ be a classical solution. 

As a first step assume $\lambda \ge 0$.
By \eqref{tondonew} and \eqref{Stimabo}, recalling the definition of $E_y$,  we have 
\[
\begin{aligned}
\ve_0&\int_s^T E_y(t)dt \le C_1E_y(s) 
+  C_2\int_s^Ty^2(t,1)dt \\
&\le (C_1+ C_2C_3)E_y(s)+  C_2 C_4\delta \int_s^T E_y(t) dt.
\end{aligned}
\]
Hence
\[
\begin{aligned}
&\left[ \ve_0 -  C_2C_4\delta\right]\int_s^T E_y(t)dt 
\le (C_1+ C_2C_3)E_y(s).
\end{aligned}
\]
Choosing $\ds
\delta\in \left( 0, \frac{\ve_0}{C_2C_4}\right),
$ we have
\[
\int_s^T E_y(t)dt  \le \frac{C_1+ C_2C_3}{\ve_0 -  C_2C_4\delta}E_y(s)
\]
for all $s \in [0, +\infty)$. Thus,
we can apply \cite[Lemma 3.1]{fm}  obtaining
\[
E_y(t) \le E_y(0) e^{1- \frac{t}{\mathcal M}},
\]
for all $t \in [\mathcal M, +\infty)$,
where  $\mathcal M$ is as in the first line of \eqref{M}.

If $\lambda <0$, then from \eqref{tondonewbis} and \eqref{Stimabo} we get
\[
\begin{aligned}
\ve_0&\int_s^T E_y(t)dt \le C_1E_y(s) 
+  C_2\int_s^Ty^2(t,1)dt -2\lambda C_{HP} \left(1+ \frac{3}{2}K_a + K_d+M\right) \int_{s}^T E_y(t)dt\\
&\le (C_1+ C_2C_3)E_y(s)+  C_2 C_4\delta \int_s^T E_y(t) dt-2\lambda C_{HP} \left(1+ \frac{3}{2}K_a + K_d+M\right) \int_{s}^T E_y(t)dt.
\end{aligned}
\]
Hence
\[
\begin{aligned}
&\left[ \ve_0 -  C_2C_4\delta+2\lambda C_{HP} \left(1+ \frac{3}{2}K_a + K_d+M\right) \right]\int_s^T E_y(t)dt 
\le (C_1+ C_2C_3)E_y(s).
\end{aligned}
\]
Proceeding as before, with the choices of $\lambda$ and $\delta$ in the statement of the theorem, one has the claim.
\end{proof}

\begin{remark}
We notice that the condition on $\lambda$ given in Theorem \ref{Energiastima} is strictly related to the one given in \cite{fms} for the null controllability result.
\end{remark}

%%%%%%%%%%%%%%%%%%%%%%%%%%%%%%%%%%%%%%%%%%%%%%%%%%%%
%%%%%%%%%%%%%%%%%%%%%%%%%%%%%%%%%%%%%%%%%%%%%%%%%%%%
\section{Appendix}
%%%%%%%%%%%%%%%%%%%%%%%%%%%%%%%%%%%%%%%%%%%%%%%%%%%%
%%%%%%%%%%%%%%%%%%%%%%%%%%%%%%%%%%%%%%%%%%%%%%%%%%%%
For the readers' convenience, in this section we prove some results used throughout the previous sections.

\subsection{Proof of Theorem \ref{thmmildclassolution}}

In order to prove Theorem  \ref{thmmildclassolution} we introduce the Hilbert space
\begin{align}\label{Hilbert0}
\mathcal H_0 := H^1_{{\frac{1}{\sigma}},0}(0,1)\times L^2_{{\frac{1}{\sigma}}}(0,1), 
\end{align}
endowed with the inner product
\[
\langle (u, v), (\tilde u, \tilde v) \rangle_{\mathcal H_0}:=  \int_0^1 u'\tilde u'dx + \int_0^1 v\tilde v\frac{1}{\sigma}dx - \lambda \int_0^1 \frac{u \tilde u}{\sigma d} dx + \beta u(1) \tilde u(1)
\]
for every $(u, v), (\tilde u, \tilde v)  \in \mathcal H_0$, which induces the norm
\[
\|(u,v)\|_{\mathcal H_0}^2:=  \int_0^1 (u')^2dx + \int_0^1 v^2\frac{1}{\sigma}dx - \lambda \int_0^1 \frac{u^2}{\sigma d} dx + \beta u^2(1).
\]
Observe that if $u \in H^1_{{\frac{1}{\sigma},0}}(0,1)$, then $u$ is continuous, so that $u(1)$ is well defined. Moreover, being $\eta \in C^0[0,1] \cap C^1(0,1]$ and bounded away from 0, the norm $\|(u,v)\|_{\mathcal H_0}^2$ is equivalent to 
\[
\|(u,v)\|_1^2:=  \int_0^1\eta (u')^2dx + \int_0^1 v^2\frac{1}{\sigma}dx - \lambda \int_0^1 \frac{u^2}{\sigma d} dx + \beta u^2(1),\quad  (u, v) \in \mathcal H_0 .
\]
Obviously, to such a norm we associate the inner product
\[
\langle (u, v), (\tilde u, \tilde v) \rangle_1:=  \int_0^1 \eta u'\tilde u'dx + \int_0^1 v\tilde v\frac{1}{\sigma}dx- \lambda \int_0^1 \frac{u \tilde u}{\sigma d} dx + \beta u(1) \tilde u(1),
\]
which we will use from now on, being more convenient for our treatment.

%Also, $H^{-1}_{{\frac{1}{\sigma}}}(0,1)$ is the dual space of $H^1_{{\frac{1}{\sigma}}}(0,1)$ with respect to the pivot space $L^2_{{\frac{1}{\sigma}}}(0,1)$. 

In order to study the well posedness of \eqref{mainequation}, we introduce the matrix operator $(\cA,  D(\cA))$ given by
\[
\cA:= \begin{pmatrix} 0 & Id\\
A_\lambda&0 \end{pmatrix},\]
\[
D(\cA):= \left\{(u,v) \in D(A_\lambda) \times H^1_{{\frac{1}{\sigma},0}}(0,1): \eta u'(1)+v(1)+ \beta u(1)=0\right\} \subset \mathcal H_0,
\]
where  $(A_\lambda, D(A_\lambda))$ is defined in \eqref{Alambda}.
 In this way,  we can rewrite \eqref{mainequation} as the Cauchy problem
\begin{equation}\label{CP}
\begin{cases}
\dot \cY (t)= \cA \cY (t), & t \ge 0,\\
\cY(0) = \cY_0,
\end{cases}
\end{equation}
with
\[
\cY(t):= \begin{pmatrix} y\\ y_t \end{pmatrix} \; \text{ and }\; \cY_0:= \begin{pmatrix} y^0_T\\ y^1_T \end{pmatrix}.
\]
\begin{thm}\label{generator}
Assume Hypotheses $\ref{hyp1}$ and $\ref{hyp2}$ and let $K_a+ 2K_d \le 2$. Then the operator $(\cA, D(\cA))$ is non positive with dense domain and generates a contraction semigroup  $(S(t))_{t \ge 0}$. 
\end{thm}
\begin{proof} The proof is based on \cite[Corollary 3.20]{nagel}.
According to this result, it is sufficient to prove that $ \mathcal A:D(\mathcal A)\to \mathcal H_0$ is dissipative and $I-\mathcal A$ is surjective.

\underline{$ \mathcal A$ is dissipative:} take $(u,v) \in D(\mathcal A)$. Then $(u,v) \in H^2_{{\frac{1}{\sigma},0}}(0,1) \times H^1_{{\frac{1}{\sigma},0}}(0,1)$ and so Lemma \ref{Lemma2.1} holds. Hence, 
\[
\begin{aligned}
\langle \mathcal A (u,v), (u,v) \rangle_{\mathcal H_0} &=\langle (v, A_\lambda u), (u,v) \rangle _{\mathcal H_0} \\
&=\int_0^1 \eta u'v'dx+ \int_0^1 vA_\lambda u\frac{1}{\sigma}dx   - \lambda \int_0^1 \frac{u v}{\sigma d} dx + \beta u(1)v(1)
\\&
=\int_0^1\eta u'v'dx -\int_0^1\eta u'v'dx + [\eta u'v](1)  + \lambda \int_0^1 \frac{u v}{\sigma d} dx- \lambda \int_0^1 \frac{u v}{\sigma d} dx\\
&+ \beta u(1)v(1)= -v^2(1) \le 0
\end{aligned}
\]
\underline{$I - \mathcal A$ is surjective:} 
take  $(f,g) \in \mathcal H_0=H^1_{{\frac{1}{\sigma},0}}(0,1)\times L^2_{{\frac{1}{\sigma}}}(0,1)$. We have to prove that there exists $(u,v) \in D(\mathcal A)$ such that
\begin{equation}\label{4.3'}
 ( I-\mathcal A)\begin{pmatrix} u\\
v\end{pmatrix} = \begin{pmatrix}f\\
g \end{pmatrix} \Longleftrightarrow  \begin{cases} v= u -f,\\
-A_\lambda u + u= f+ g.\end{cases}
\end{equation}
Thus, define $F: H^1_{{\frac{1}{\sigma},0}}(0,1) \rightarrow \R$ as
\[
F(z)=\int_0^1(f+g) z\frac{1}{\sigma}  dx + z(1)f(1).
\]
Obviously, $F\in H^{-1}_{{\frac{1}{\sigma},0}}(0,1)$, the dual space of $H^1_{{\frac{1}{\sigma},0}}(0,1)$ with respect to the pivot space $L^2_{{\frac{1}{\sigma}}}(0,1)$. Now, introduce the bilinear form $L:H^1_{{\frac{1}{\sigma},0}}(0,1)\times H^1_{{\frac{1}{\sigma},0}}(0,1)\to \R$ given by
\[
L(u,z):=  \int_0^1 u z \frac{1}{\sigma} dx + \int_0^1\eta u'z'dx -\lambda \int_0^1  \frac{u z}{\sigma d} dx+ (\beta+1) u(1)z(1)
\]
for all $u, z \in H^1_{{\frac{1}{\sigma},0}}(0,1)$. Clearly, $L(u,z)$ is coercive: indeed, for all $u \in  H^1_{{\frac{1}{\sigma},0}}(0,1)$, we have 
\[
\begin{aligned}
L(u,u) &= \int_0^1  \frac{u^2}{\sigma} dx + \int_0^1\eta (u')^2dx -\lambda \int_0^1  \frac{u^2}{\sigma d} dx + (\beta+1)u^2(1)\\
&\ge  \int_0^1  \frac{u^2}{\sigma} dx + \int_0^1\eta (u')^2dx = \|u\|^2_{1, \frac{1}{\sigma}},
\end{aligned}
\]
if $\lambda \le 0$. On the other hand,   if $\lambda \in (0, C_{HP})$ then
\[
\begin{aligned}
L(u,u) \ge   \int_0^1  \frac{u^2}{\sigma} dx + (1-\lambda C_{HP})\int_0^1\eta (u')^2dx  \ge (1-\lambda C_{HP})\|u\|^2_{1, \frac{1}{\sigma}},
\end{aligned}
\]
by \eqref{stima1}.
 Moreover $L(u,z)$ is 
 continuous: indeed, using again \eqref{stima1}, one has
\[
\begin{aligned}
|L(u,z)| &\le  \|u\|_{ L^2_{\frac{1}{\sigma}} (0,1) }\|z\|_ {L^2_{\frac{1}{\sigma}} (0,1) } \\
&+(\|\eta\|_{L^\infty(0,1)}+ |\lambda|\|\eta\|_{L^\infty(0,1)}C_{HP}+ (\beta+1))\|u'\|_{L^2(0,1)}\|z'\|_ {L^2(0,1)},
\end{aligned}\]
 for all $u, z \in H^1_{{\frac{1}{\sigma},0}}(0,1)$, being
 \[
 |u(1)| \le \int_0^1 |u'(x)|dx \le \|u'\|_{L^2(0,1)}
 \]
 (analogously for $z(1)$),
 and the conclusion follows.

As a consequence, by the Lax-Milgram Theorem, there exists a unique solution $u \in H^1_{{\frac{1}{\sigma},0}}(0,1)$ of
\[
L(u,z)= F(z) \quad  \mbox{ for all }z\in H^1_{{\frac{1}{\sigma},0}}(0,1),\]
namely
\begin{equation}\label{4.4}
\int_0^1  \frac{u z}{\sigma} dx + \int_0^1 \eta u'z'dx -\lambda \int_0^1  \frac{u z}{\sigma d} dx +(\beta+1) u(1)z(1)
= \int_0^1(f+g) z \frac{1}{\sigma} dx  + z(1)f(1)
\end{equation}
for all $z \in H^1_{{\frac{1}{\sigma},0}}(0,1)$.

Now, take $v:= u-f$; then $v \in H^1_{{\frac{1}{\sigma},0}}(0,1)$.
We will prove that $(u,v) \in D(\mathcal A)$ and solves \eqref{4.3'}. To begin with, \eqref{4.4} holds for every $z \in C_c^\infty(0,1).$ Thus we have
\[
\int_0^1 \eta u'z'dx = \int_0^1\left(f+g-u+ \lambda \frac{u}{d}\right) z \frac{1}{\sigma} dx 
\]
 for every $z \in C_c^\infty(0,1).$ Hence $\ds-(\eta u')'= \left(f+g-u+ \lambda \frac{u}{d}\right) \frac{1}{\sigma}$ a.e. in $(0,1)$, i.e. $-\sigma(\eta u')'=\left(f+g-u+ \lambda\ds \frac{u}{d}\right) $ a.e. in $(0,1)$. In particular, $A_\lambda u= Au+ \lambda\ds \frac{u}{d}=-\left(f+g-u\right) \in L^2_{ \frac{1}{\sigma}}(0,1)$. Moreover, using the assumption $K_a+ 2K_d \le 2$, one has that $\ds \frac{u}{d}  \in L^2_{ \frac{1}{\sigma}}(0,1)$; thus $u \in D(A_\lambda)$. Coming back to \eqref{4.4}  and thanks to Lemma \ref{Lemma2.1}, one has 
 \[
 - \int_0^1 \sigma(\eta u')'z \frac{1}{\sigma} dx + (\eta u' z)(1)  +(\beta+1) u(1)z(1)
= \int_0^1\left(f+g-u+ \lambda\ds \frac{u}{d}\right) z \frac{1}{\sigma} dx  + z(1)f(1)
\] and
 \[
 (\eta u' z)(1)  +(\beta+1) u(1)z(1)=z(1)f(1)
 \]
for all $z \in H^1_{{\frac{1}{\sigma},0}}(0,1)$, since  $-\sigma(\eta u')'=\left(f+g-u+ \lambda\ds \frac{u}{d}\right) $ a.e. in $(0,1)$, . Hence 
$
\eta (1) u' (1) +(\beta+1) u(1) - f(1)=0 .
$
Recalling that $v = u - f$, one has
\[
\eta (1) u' (1) +\beta u(1) + v(1)=0.
\]
In conclusion,  $(u,v) \in D(\mathcal A)$
 and \eqref{4.3'} holds.
\end{proof}

As usual in semigroup theory, the mild solution of \eqref{CP} obtained above can be more regular: if $\cY_0 \in D(\mathcal A)$, then the solution is classical, in the sense that $\cY \in  C^1([0, +\infty); \mathcal H_0) \cap C([0, +\infty);D(\mathcal A))
$ and the equation in \eqref{mainequation} holds for all $t \ge0$. Hence, as in \cite[Corollary 4.2]{alabau} or in \cite[Proposition 3.15]{daprato}, one can deduce Theorem \ref{thmmildclassolution}.
\section{Aknowledgments}
G. Fragnelli is partially supported by INdAM GNAMPA Project 2023 {\it ``Modelli differenziali per l'evoluzione del clima e i suoi impatti"} (CUP E53C22001930001), by INdAM GNAMPA Project 2024 {\it ``Analysis, control and inverse problems for evolution
		equations arising in climate science"} (CUP E53C23\-001670001) and by the PRIN 2022 PNRR project {\it Some Mathematical approaches to climate change and its impacts} (CUP E53D23017910001). She is also a member of {\it UMI ``Modellistica Socio-Epidemiologica (MSE)''}.
		\noindent D. Mugnai  is partially supported by INdAM GNAMPA Project 2023 ”Variational and non-variational problems with lack of compactness ”
(CUP E53C22001930001), by INdAM GNAMPA Project 2024 ”Nonlinear
problems in local and nonlocal settings with applications" (CUP E53C23001670001)
and by the PRIN 2022 ”Advanced theoretical aspects in PDEs and their applications” (CUP J53D23003700006).
		\noindent G. Fragnelli and D. Mugnai are members of {\it UMI "CliMath"} and  of  {\it Gruppo Nazionale per l'Analisi Ma\-te\-matica, la Probabilit\`a e le loro Applicazioni (GNAMPA)} of the Istituto Nazionale di Alta Matematica (INdAM). 
	
	\noindent This paper was partially written during the visit of A. Sbai at Università of Tuscia, under a grant of the 
MAECI (Ministry of Foreign Affairs and International Cooperation, Italy).

%%%%%%%%%%%%%%%%%%%%%%%%%%%%%%%%%%%%%%%%%%%%%%%%%%%%%%%%%%%%%%%%%%%%%%%%%%%%%%%%%%%%%%%%%%%%%%%%%%%%%%%%%%

\end{document}